\documentclass[12pt,a4paper,reqno]{amsart}
\usepackage{amsmath, amsthm, amsfonts, amssymb}
\usepackage{hyperref}

\numberwithin{equation}{section}

\theoremstyle{plain}
\newtheorem{theorem}[subsection]{Theorem}
\newtheorem{proposition}[subsection]{Proposition}
\newtheorem*{prop-no-number}{Proposition}
\newtheorem*{prop-}{Proposition}
\newtheorem{lemma}[subsection]{Lemma}
\newtheorem{corollary}[subsection]{Corollary}

\theoremstyle{definition}
\newtheorem{defn}[subsection]{Definition}

\newtheorem{remark}[subsection]{Remark}

\renewcommand{\leq}{\leqslant}
\renewcommand{\geq}{\geqslant}
\renewcommand{\subset}{\subseteq}

\def\E{\mathbb{E}}
\def\Z{\mathbb{Z}}

\def\T{\mathbb{T}}
\def\C{\mathbb{C}}
\def\N{\mathbb{N}}

\def\F{\mathbb{F}}

\def\Ghat{\widehat{G}}
\def\fhat{\widehat{f}}
\newcommand{\ud}{\,\mathrm{d}}

\DeclareMathOperator{\Bohr}{Bohr}

\newcommand{\Zmod}[1]{\Z_{#1}} 

\providecommand{\abs}[1]{\lvert#1\rvert}
\providecommand{\norm}[1]{\lVert #1 \rVert}

\providecommand{\ceiling}[1]{\lceil#1\rceil}
\providecommand{\floor}[1]{\left\lfloor#1\right\rfloor}
\providecommand{\fracpart}[1]{\left\{#1\right\}}

\theoremstyle{plain}
\makeatletter
\newtheorem*{rep@theorem}{\rep@title}
\newcommand{\newreptheorem}[2]{%
\newenvironment{rep#1}[1]{%
\def\rep@title{#2 \ref*{##1}}%
\begin{rep@theorem}}%
{\end{rep@theorem}}}
\makeatother

\newreptheorem{theorem}{Theorem}

\title{Arithmetic progressions in sumsets and $L^p$-almost-periodicity}
\author{Ernie Croot}
\address{School of Mathematics\\
     Georgia Institute of Technology\\
     Atlanta, GA 30332\\
     USA
}
\email{ecroot@math.gatech.edu}
\author{Izabella {\L}aba}
\address{Department of Mathematics\\
     University of British Columbia\\
     Vancouver\\
     B.C. V6T 1Z2\\
     Canada
}
\email{ilaba@math.ubc.ca}
\author{Olof Sisask}
\address{School of Mathematical Sciences\\
     Queen Mary, University of London\\
     Mile End Road\\
     London E1 4NS\\
     United Kingdom
}
\email{O.Sisask@qmul.ac.uk}


\begin{document}

\begin{abstract}
We prove results about the $L^p$-almost-periodicity of convolutions. One of these follows from a simple but rather general lemma about approximating a sum of functions in $L^p$, and gives a very short proof of a theorem of Green that if $A$ and $B$ are subsets of $\{1,\ldots,N\}$ of sizes $\alpha N$ and $\beta N$ then $A+B$ contains an arithmetic progression of length at least
\[ \exp\left( c (\alpha \beta \log N)^{1/2} - \log\log N \right). \]
Another almost-periodicity result improves this bound for densities decreasing with $N$: we show that under the above hypotheses the sumset $A+B$ contains an arithmetic progression of length at least
\[ \exp\left( c \left(\frac{\alpha \log N}{\log^3 2\beta^{-1}} \right)^{1/2} - \log( \beta^{-1} \log N) \right). \]
\end{abstract}

\maketitle

\parindent 0mm
\parskip   4mm

\section{Introduction}
Let $A$ and $B$ be subsets of an abelian group. What can one say about the structure of the sumset $A+B$? One of the main endeavours of additive combinatorics is to answer questions along these lines, and one particularly appealing goal is to find long arithmetic progressions in $A+B$ for particular kinds of sets $A$ and $B$. Under the assumption that $A$ and $B$ are `large', a remarkable result in this direction was established by Bourgain \cite{bourgain:longAPs}, who showed that if $A$ and $B$ are subsets of $\{1,\ldots,N\}$ of sizes $\alpha N$ and $\beta N$, then $A+B$ contains an arithmetic progression of length about\footnote{Here, and throughout the paper, we employ the convenient device of letting the letters $c$ and $C$ denote positive absolute constants that may change from occurrence to occurrence.} $\exp( c (\alpha \beta \log N)^{1/3})$. This was subsequently improved by Green \cite{green:longAPs} to the following.

\begin{theorem}[Green]\label{thm:green}
Suppose $A$ and $B$ are subsets of $\{1,\ldots,N\}$ with densities $\alpha$ and $\beta$. Then $A+B$ contains an arithmetic progression of length at least
\[ \exp\left( c (\alpha \beta \log N)^{1/2} - \log\log N\right). \]
\end{theorem} 

Green's method of proof was quite different to Bourgain's, which involved establishing a particular kind of almost-periodicity result, and yet a third strategy was employed by Sanders \cite{sanders:longAPs} in giving an alternative proof of the theorem. In this paper we prove almost-periodicity results along similar lines to Bourgain's and those in \cite{croot-sisask} and give applications in the direction of Theorem \ref{thm:green}. Let us first state a particularly clean almost-periodicity result; this yields precisely Green's bounds when applied in the direction of finding arithmetic progressions in sumsets.

\begin{theorem}\label{thm:fourier_lp}
Let $p \geq 2$ and $\epsilon \in (0,1)$ be parameters. Let $G$ be a finite abelian group and suppose $f : G \to \C$ is a function. Then there is a Bohr set $T \subset G$ of rank at most $Cp/\epsilon^2$ and radius $c\epsilon$ such that, for each $t \in T$,
\[ \norm{ f(x+t) - f(x) }_{L^p(x)} \leq \epsilon \norm{\fhat}_{\ell^1}. \]
\end{theorem}

(We give definitions in the next section.) Taking $f = 1_A*1_B$, the proof of Theorem \ref{thm:green} from this result is very quick. What is perhaps surprising is that the proof of this theorem is also very short and involves hardly any Fourier analysis. Instead it relies on a simple and quite general probabilistic argument similar to that employed in \cite{croot-sisask}; this is contained in Section \ref{section:prob_arg}.

Theorem \ref{thm:fourier_lp} is quite general and makes no assumptions on $f$; however, in order for it to be effective we must assume that $\norm{\fhat}_{\ell^1}$ is sufficiently small.  This will indeed be the case in the application to proving Theorem \ref{thm:green}, where $f = 1_A*1_B$ and $\norm{\fhat}_{\ell^1}$ is well controlled by the Cauchy-Schwarz inequality.

By a different and somewhat more involved argument we are also able to establish the following strengthening of Theorem \ref{thm:green}.

\begin{theorem}\label{thm:impAPs}
Suppose $A$ and $B$ are subsets\footnote{We shall always assume that $A$ and $B$ are non-empty, and that $N \geq 3$, to avoid trivialities.} of $\{1,\ldots,N\}$ with densities $\alpha$ and $\beta$. Then $A+B$ contains an arithmetic progression of length at least
\[ \exp\left( c \left(\frac{\alpha \log N}{(\log 2\beta^{-1})^3} \right)^{1/2} - \log( \beta^{-1} \log N) \right). \]
\end{theorem}

The improvement of this over Theorem \ref{thm:green} enters when one of the sets has density decreasing with $N$. Note, for example, that for Theorem \ref{thm:green} to yield a non-trivial conclusion at least one of the sets needs to have density at least $\log \log N/(\log N)^{1/2}$, and each set density at least $(\log \log N)^2/\log N$, whereas Theorem \ref{thm:impAPs} allows for both sets to have density about $(\log \log N)^C/\log N$, and in fact one of the sets can have density as low as $\exp\left(- (\log N)^{c}\right)$.

We shall in fact deduce this theorem from the following small-doubling version, where the sets involved are not assumed to be dense in an interval but rather satisfy a small sumset hypothesis.

\begin{theorem}\label{thm:impAPsDoubling}
Suppose $A$ and $B$ are sets of integers satisfying $\abs{A+B} \leq K_A\abs{A}, K_B\abs{B}$. Then $A+B$ contains an arithmetic progression of length at least
\[ \exp\left( c \left(\frac{\log{\abs{A+B}}}{K_B (\log 2K_A)^3} \right)^{1/2} - \log(2K_A\log{2\abs{A}}) \right). \]
\end{theorem}

Theorem \ref{thm:impAPs} follows on noting that if $A$ and $B$ are subsets of $\{1,\ldots,N\}$ then $\abs{A+B} \leq 2N$, whence we may take $K_A = 2\alpha^{-1}$ and $K_B = 2\beta^{-1}$. 

Behind this theorem lies another almost-periodicity result; this time the driving force is a combination of the main result from \cite{croot-sisask} with an observation from the remarkable paper \cite{sanders:bogolyubov} of Sanders. The result, Theorem \ref{thm:strong_lp}, is more naturally stated in the context of other results and so we postpone it till Section \ref{section:strong_lp}.

The remainder of this paper is laid out as follows. In the next section we fix notation and note some preliminary lemmas. In Section \ref{section:prob_arg} we prove a general lemma on approximating functions in $L^p$; this forms the heart of our proof of Theorem \ref{thm:fourier_lp}, which we deduce together with Theorem \ref{thm:green} in Section \ref{section:simple_periodicity}. Together these sections provide all the material needed for the reader interested purely in a short proof of Theorem \ref{thm:green}. The proof of Theorem \ref{thm:impAPsDoubling} proceeds by different means, starting in Section \ref{section:model} where we note a variant of Ruzsa's model lemma that works efficiently with two sets. In Section \ref{section:strong_lp} we state and prove Theorem \ref{thm:strong_lp}, the almost-periodicity result required for Theorem \ref{thm:impAPsDoubling}, and deduce the latter theorem. Finally, in Section \ref{section:remarks} we make some further remarks, including details about the finite field versions of our results.

\subsection*{Acknowledgement}
The first author was supported by NSF grant DMS-1001111, the second author by an NSERC Discovery Grant, and the third author in parts by an NSERC Discovery Grant and an EPSRC Postdoctoral Fellowship. The third author would also like to acknowledge the hospitality of the University of British Columbia, where part of this work was done.

\section{Notation, definitions and preliminaries}
We now fix notation and give definitions. Most of what we use is standard in the additive combinatorial literature; as a consequence we shall be rather brief, referring the unfamiliar reader to the book \cite{tao-vu} of Tao and Vu for more information about any of the following notions.

For a finite abelian group $G$ we write $\Ghat = \{ \gamma : G \to \C^\times : \text{$\gamma$ a homomorphism} \}$ for its dual group, and we define the Fourier transform $\fhat : \Ghat \to \C$ of a function $f : G \to \C$ by $\fhat(\gamma) = \E_{x \in G} f(x) \overline{\gamma(x)}$. Here, and elsewhere in this paper, the notation $\E_{x \in X} = \abs{X}^{-1}\sum_{x \in X}$ denotes the average over a finite set $X$; when $X$ is clear from the context we may write simply $\E_x$. We define the \emph{convolution} of two functions $f, g : G \to \C$ by $f*g(x) = \E_y f(y) g(x-y)$. With these normalizations, the Fourier inversion formula, Parseval's identity and the convolution identity take the form
\begin{gather*}
f(x) = \sum_{\gamma \in \Ghat} \fhat(\gamma) \gamma(x), \\
\E_{x \in G} \abs{f(x)}^2 = \sum_{\gamma \in \Ghat} \abs{\fhat(\gamma)}^2, \text{ and}\\
\widehat{f*g}(\gamma) = \fhat(\gamma) \widehat{g}(\gamma).
\end{gather*}

We plainly need the notion of a Bohr set; our definition here is not identical to that in \cite{tao-vu} but is essentially equivalent.

\begin{defn}
Let $G$ be a finite abelian group. For a set $\Gamma \subset \Ghat$ of characters and a number $\delta \geq 0$ we define
\[ \Bohr_G(\Gamma, \delta) = \{ x \in G : \abs{\gamma(x) - 1} \leq \delta \text{ for all $\gamma \in \Gamma$} \} \]
and call this\footnote{There is an abuse of notation here; technically one should specify the triple $(G,\Gamma,\delta)$ and not just the set $\Bohr_G(\Gamma, \delta)$ in order to discuss the rank and radius of a Bohr set.} a \emph{Bohr set} of \emph{rank} $\abs{\Gamma}$ and \emph{radius} $\delta$.
\end{defn}

A Bohr set is thus a set of approximate annihilators of a set of characters, and it is these days well-known that such sets play an important role in additive combinatorics. Of critical importance is that Bohr sets are highly structured:

\begin{lemma}\label{lemma:bohr_sets}
Suppose $T$ is a Bohr set of rank $d \geq 1$ and radius $\delta \in (0,2]$ in a finite abelian group $G$. Then $\abs{T} \geq (\delta/2\pi)^d \abs{G}$. Furthermore, if $G = \Zmod{N}$ is a group of residues modulo a prime $N$, then $T$ contains an arithmetic progression of size at least $\frac{1}{2\pi}\delta N^{1/d}$.
\end{lemma}

See for example \cite[\S4.4]{tao-vu} for proofs.

For a finite abelian group $G$ we write $\mu_G$ for the normalized counting measure $\mu_G(A) = \abs{A}/\abs{G}$ on $G$, and for a subset $X$ of $G$ we write $\mu_X = 1_X/\mu_G(X)$. For a function $f : X \to \C$ we use the notation $\norm{f}_{\ell^p} = (\sum_{x \in X} \abs{f(x)}^p)^{1/p}$ for the $\ell^p$ norm of $f$ and $\norm{f}_{L^p} = \norm{f}_p = (\E_x \abs{f(x)}^p)^{1/p}$ for the $L^p$ norm of $f$. 

Finally, we write $\T = \{ z \in \C : \abs{z} = 1 \}$ for the unit circle in $\C$.

\section{Approximating a linear combination of functions in $L^p$}\label{section:prob_arg}
In this section we present a rather general lemma about approximating a linear combination of functions by a combination of just a few of those functions. Theorem \ref{thm:fourier_lp}, the simple almost-periodicity result, will be a simple deduction from this lemma. To illustrate the lemma, let us sketch how it works as applied in the proof of Theorem \ref{thm:fourier_lp}. To prove Theorem \ref{thm:fourier_lp} it suffices to approximate $f$ in $L^p$ by a linear combination of a few characters. The Fourier inversion formula gives us an expression
\begin{align} f = \sum_{\gamma \in \Ghat} \fhat(\gamma) \gamma \label{eqn:FI} \end{align}
for $f$ as a linear combination of characters, but too many to be of use. The idea, then, is to reduce the number of characters used by sampling some at random. According to what distribution ought we to sample? An answer is suggested directly by \eqref{eqn:FI}: one can view this expression essentially as the expectation of a random character $\chi$, where $\chi$ is picked to be the character $\gamma$ with probability proportional to $\fhat(\gamma)$ (ignoring technicalities). Taking an average of a few such random characters then approximates $f$ well in $L^p$ by the law of large numbers, yielding what we want.

The version of the law of large numbers we shall use is the following inequality of Marcinkiewicz and Zygmund.

\begin{lemma}[Marcinkiewicz-Zygmund inequality]
Suppose $X_1, \ldots, X_n$ are independent, mean zero (complex-valued) random variables with $\E \abs{X_i}^p < \infty$. Then
\[ \E \abs{ \sum_{j=1}^k X_j }^p \leq (Cp)^{p/2} \E \left( \left( \sum_{j=1}^k \abs{X_j}^2 \right)^{p/2} \right).\]
\end{lemma}
This is proved in many places, though attention is often not paid to the dependence on $p$ in the multiplicative constant. One route to get this particular dependence on $p$ is to note that one has this form for the bound for Rademacher random variables via Khintchine's inequality \cite{khintchine}, and that one can deduce the result for general real-valued random variables from this by the techniques of symmetrization and randomization; see for example \cite[\S 3.8]{gut}. The complex-valued case follows easily from the real case.

With this in hand, the following lemma implements the idea outlined above in a slightly more general setup. For a non-zero complex number $\lambda$ we write $\lambda^{\circ} = \lambda/\abs{\lambda} \in \T$ for the direction on the unit circle determined by $\lambda$, and we set $0^{\circ} = 0$.

\begin{lemma}\label{lemma:approxLp}
Let $p \geq 2$ and $\epsilon \in (0,1)$ be parameters, and let $(X,\mu)$ be a probability space. Suppose
\[ f = \lambda_1 g_1 + \cdots + \lambda_N g_N \]
where each $\lambda_j \in \C$ and each $g_j : X \to \C$ is a measurable function with $\norm{g_j}_{L^p(\mu)} \leq 1$. Then there is a positive integer $k \leq Cp/\epsilon^2$ and a $k$-tuple $\sigma \in [N]^k$ such that
\[ \norm{ \tfrac{1}{\norm{\lambda}_{\ell^1}}f - \tfrac{1}{k}( \lambda_{\sigma_1}^{\circ} g_{\sigma_1} + \cdots + \lambda_{\sigma_k}^{\circ} g_{\sigma_k} ) }_{L^p(\mu)} \leq \epsilon. \]
In fact, the inequality holds with probability at least $0.99$ if the tuple $(\sigma_1, \ldots, \sigma_k)$ is picked with probability $\abs{\lambda_{\sigma_1} \cdots \lambda_{\sigma_k}}/\norm{\lambda}_{\ell^1}^k$.
\end{lemma}
\begin{remark}
In the applications we have in mind one generally has $X$ being a finite set and $\mu$ being uniform measure on some subset of $X$.
\end{remark}
\begin{proof}
Let $h$ be picked randomly from $\lambda_1^{\circ} g_1, \ldots, \lambda_N^{\circ} g_N$, with $\lambda_j^{\circ} g_j$ being picked with probability $\abs{\lambda_j}/\norm{\lambda}_{\ell^1}$. Then we have
\[ \E h = \sum_{j=1}^N \frac{\lambda_j}{\norm{\lambda}_{\ell^1}} g_j = f/\norm{\lambda}_{\ell^1} =: f_0. \]
Let $h_1,\ldots,h_k$ be iid copies of $h$. The average $k^{-1}(h_1 + \cdots + h_k)$ ought then to approximate its expectation $f_0$ provided $k$ is not too small; and indeed, by the Marcinkiewicz-Zygmund inequality coupled with two interchanges of orders of integration we have
\begin{align}
\E \int_{X} \abs{ \frac{1}{k} \sum_{j=1}^k h_j(x) - f_0(x) }^p \ud\mu(x) &\leq \frac{(Cp)^{p/2}}{k^{p/2}} \E \int_{X} \left(\frac{1}{k}\sum_{j=1}^k \abs{ h_j(x) - f_0(x) }^2 \right)^{p/2} \ud\mu(x) \nonumber\\
&\leq \frac{(Cp)^{p/2}}{k^{p/2}}, \label{eqn:M-Z}
\end{align}
the second inequality following from the nesting of $L^q$ norms and the assumption that $\norm{g_j(x)}_{L^p(\mu)} \leq 1$ for each $j$. Picking $k = \lceil{Cp/\epsilon^2}\rceil$ gives that this is at most $c\epsilon^p$ for some small constant $c$, and so
\[ \E \norm{\frac{1}{k} \sum_{j=1}^k h_j(x) - f_0(x)}_{L^p(\mu)}^p \leq c\epsilon^p. \]
Hence, by Markov's inequality, we have that
\[ \norm{\frac{1}{k} \sum_{j=1}^k h_j(x) - f_0(x)}_{L^p(\mu)} \leq \epsilon \]
holds with probability at least $1-c$, as required.
\end{proof}

\begin{remark}
In obtaining \eqref{eqn:M-Z} we bounded the integrals rather crudely using the bounds $\norm{h_j(x)}_{L^p(\mu)}, \norm{f_0}_{L^p(\mu)} \leq 1$. One can in certain situations make some additional savings here; see the remarks at the end of the next section.
\end{remark}

\begin{remark}
Since the completion of this paper we learned that abstract random sampling ideas similar to the above have been used in other contexts before. In particular, our proof of Lemma \ref{lemma:approxLp} essentially constitutes a proof of Maurey's lemma in Banach space theory as recorded in \cite{pisier:Maurey}. We thank Assaf Naor for directing us to this paper.
\end{remark}

Applied to the Fourier expansion \eqref{eqn:FI}, the lemma has the following immediate corollary.

\begin{corollary}\label{cor:FourierSampling}
Let $G$ be a finite abelian group and let $\epsilon \in (0,1)$ and $p \geq 2$ be parameters. For any non-zero function $f : G \to \C$ there are characters $\gamma_1, \ldots, \gamma_k \in \Ghat$ with $k \leq Cp/\epsilon^2$ and coefficients $c_1, \ldots, c_k \in \C$ with $\abs{c_j} = 1$ such that
\[ \norm{ f/\norm{\fhat}_{\ell^1} - \tfrac{1}{k}(c_1 \gamma_1 + \cdots + c_k \gamma_k) }_{L^p} \leq \epsilon. \]
\end{corollary}

\section{Deducing almost-periodicity and finding progressions in sumsets}\label{section:simple_periodicity}
We are now ready to prove Theorems \ref{thm:green} and \ref{thm:fourier_lp}. We restate these here so that this section provides a self-contained account of Green's bounds, modulo Lemma \ref{lemma:bohr_sets} and Corollary \ref{cor:FourierSampling}.

\begin{reptheorem}{thm:fourier_lp}
Let $p \geq 2$ and $\epsilon \in (0,1)$ be parameters. Let $G$ be a finite abelian group and suppose $f : G \to \C$ is a function. Then there is a Bohr set $T \subset G$ of rank at most $Cp/\epsilon^2$ and radius $c\epsilon$ such that, for each $t \in T$,
\[ \norm{ f(x+t) - f(x) }_{L^p(x)} \leq \epsilon \norm{\fhat}_{\ell^1}. \]
\end{reptheorem}
\begin{proof}
Applying Corollary \ref{cor:FourierSampling} gives us characters $\gamma_1, \ldots, \gamma_k$ with $k \leq Cp/\epsilon^2$ and complex numbers $c_1, \ldots, c_k$ with $\abs{c_j} = 1$ such that $g = (c_1 \gamma_1 + \cdots + c_k \gamma_k)/k$ satisfies
\[ \norm{ f/\norm{\fhat}_{\ell^1} - g }_{L^p} \leq \epsilon/3. \]
But $g$ is $L^{\infty}$-almost-periodic: for any $t \in \Bohr_G(\{ \gamma_1, \ldots, \gamma_k \}, \epsilon/3)$, the triangle inequality yields
\[ \abs{g(x+t) - g(x)} \leq \frac{1}{k}\sum_{j=1}^k \abs{ \gamma_j(t) - 1 } \leq \epsilon/3 \]
for any $x \in G$, whence the theorem follows by the $L^p$ triangle inequality.
\end{proof}

We can now deduce Theorem \ref{thm:green}.

\begin{reptheorem}{thm:green}
Suppose $A$ and $B$ are subsets of $\{1,\ldots,N\}$ with densities $\alpha$ and $\beta$. Then $A+B$ contains an arithmetic progression of length at least
\[ \exp\left( c (\alpha \beta \log N)^{1/2} - \log\log N\right). \]
\end{reptheorem}
\begin{proof}
By embedding the interval $\{1,\ldots,N\}$ in the cyclic group $\Zmod{N'}$, where $4N \leq N' \leq 8N$ is a prime, it suffices to establish the theorem for cyclic groups of prime order instead of intervals; we shall thus prove the theorem for $A, B \subset \Zmod{N}$ with densities $\alpha, \beta$.

Let $f = 1_A*1_B$ and apply Theorem \ref{thm:fourier_lp} to $f$ with parameters to be determined; this gives us a Bohr set $T$ of rank $d \leq Cp/\epsilon^2$ and radius $c\epsilon$ such that, for any $t \in T$,
\[ \norm{ 1_A*1_B(x+t) - 1_A*1_B(x) }_{L^p(x)} < \epsilon \norm{\widehat{1_A} \widehat{1_B}}_{\ell^1} \leq \epsilon (\alpha \beta)^{1/2}, \]
where we have used the convolution identity, the Cauchy-Schwarz inequality and Parseval's identity. Let $P$ be a long arithmetic progression in $T$; we shall determine how long we can pick it in a moment. We then have
\begin{align*}
\E_{x \in G} \sup_{t \in P} \abs{ 1_A*1_B(x+t) - 1_A*1_B(x) } &\leq \E_x \left( \sum_{t \in P} \abs{ 1_A*1_B(x+t) - 1_A*1_B(x) }^p \right)^{1/p} \\
&< \abs{P}^{1/p} \epsilon (\alpha \beta)^{1/2}
\end{align*}
by the nesting of $L^p$ norms. If this were less than $\alpha \beta = \E_x 1_A*1_B(x)$ then we would be done, for then there would be an element $x$ such that $1_A*1_B(x+t) > 0$ for each $t \in P$, whence $x+P \subset A+B$. Let us pick $\epsilon = (\alpha \beta)^{1/2}/e$; we shall then be done if $\abs{P} \leq e^p$. By Lemma \ref{lemma:bohr_sets} we can find an arithmetic progression $P$ in $T$ of length the integer part of
\[ c\epsilon N^{1/d} = \exp\left( c p^{-1} \alpha \beta \log N - \log(C (\alpha \beta)^{-1} ) \right), \]
so we pick $p = C \sqrt{ \alpha \beta \log N }$ in order to ensure that $\abs{P} \leq e^p$. Noting that the theorem is trivial if $\alpha \beta \leq C/\log N$, we are done. 
\end{proof}

\subsection*{Fourier space versus physical space}
We should note that one can also obtain physical-side almost-periodicity results via Lemma \ref{lemma:approxLp} by applying it to the expansion
\[ 1_A*1_B = \sum_{y \in G} \frac{1_A(y)}{\abs{G}} 1_{y+B}. \]
So applied, the lemma essentially lets one approximate $1_A*1_B$ by a multiple of $1_C*1_B$ for many small sets $C \subset A$, and from this one can deduce almost-periodicity results: see \cite{croot-sisask}. One can in fact obtain slightly stronger estimates for the expectations involved in the proof of Lemma \ref{lemma:approxLp} in this setup than those concluded above. In particular, by specializing the proof to this expansion at \eqref{eqn:M-Z}, one naturally obtains the following version of the main result of \cite{croot-sisask}; see also \cite[Lemma 3.3]{sanders:bogolyubov}.

\begin{theorem}\label{thm:plain_lp}
Let $p \geq 2$, $\epsilon \in (0,1)$ and $k \in \N$ be parameters. Let $G$ be a group and let $A$, $B$ and $S$ be finite subsets of $G$ with $\abs{S\cdot A} \leq K\abs{A}$. Then there is a set $T \subset S$ of size
\[ \abs{T} \geq \frac{0.99\abs{S}}{K^{C p k^2/\epsilon^2}} \]
such that, for each $t \in (T^{-1}T)^k$,
\[ \norm{ 1_A*1_B(tx) - 1_A*1_B(x) }_{\ell^p(x)} \leq \epsilon \abs{A} \abs{B}^{1/p}. \]
\end{theorem}

Here we have used product notation for the group operation and unnormalized convolutions $f*g(x) = \sum_{y\in G} f(y)g(y^{-1}x)$.

By making use of more detailed still distributional information in the random sampling, one can obtain estimates of a different form for the almost-periodicity of $1_A*1_B$. This was done in \cite{croot-sisask}, and it is the use of a particular form of these estimates that leads to the improved density dependence in Theorem \ref{thm:impAPs}.

\section{A variant of Ruzsa's model lemma}\label{section:model}
We now turn to the proof of Theorem \ref{thm:impAPsDoubling}, which finds long arithmetic progressions in $A+B$ if $A$ and $B$ are sets of integers with $\abs{A+B}$ small. Before we can embark on the proof proper, we need to transfer the problem to one in a cyclic group of prime order, where we have more effective tools at our disposal. If $A$ and $B$ were assumed to be dense subsets of $\{1,\ldots,N\}$, then we could simply embed the sets in $\Zmod{N'}$ for a small prime $N'$ and work there instead of in the integers, as in the proof of Theorem \ref{thm:green}. In the small-doubling setup, however, we need to argue more subtly; the following modification of a modelling lemma of Ruzsa \cite[Theorem 2.3.5]{ruzsa:sumsets_structure} will allow us to proceed efficiently.

\begin{proposition}\label{prop-ruzsa}
Let $A$ and $B$ be finite sets of integers and suppose $k \geq 2$ is an integer. Then for any integer $N \geq \abs{kA-kA+kB-kB}$ there are subsets $A' \subset A$ and $B'\subset B$ with $\abs{A'} \geq \abs{A}/2k$, $\abs{B'} \geq \abs{B}/2k$ such that $A'+B'$ is $k$-isomorphic to a subset of $\Zmod{N}$.
\end{proposition}
In this context, two subsets $A$ and $A'$ of two abelian groups are said to be \emph{Freiman $k$-isomorphic} if there is a bijection $\phi : A \to A'$ such that
\[ a_1 + \cdots + a_k = a_1' + \cdots + a_k' \Longleftrightarrow \phi(a_1) + \cdots + \phi(a_k) = \phi(a_1') + \cdots + \phi(a_k') \]
whenever $a_j \in A$; such a map is called a \emph{$k$-isomorphism}. The relevance of this property here is that a $2$-isomorphism preserves the structure of a sumset, as well as the structure of an arithmetic progression.

In the following proof we use the notation $\fracpart{x} = x - \floor{x}$ to represent the fractional part of a real number $x$.
\begin{proof}
We follow Ruzsa's proof with minor modifications. Let $\xi$ be a real number in the interval $[0,N]$, to be fixed later, and define the map $\phi : \Z \to \Zmod{N}$ by
\[ \phi(a) = \floor{ \xi a } \pmod{N}. \]
We shall show that if $\xi$ is picked appropriately then one can find $A'$ and $B'$ as required such that the map $\psi(a+b) := \phi(a) + \phi(b)$ is well-defined and a $k$-isomorphism on $A'+B'$.

To this end, let us define
\[ A_j = \left\{ a \in A : \frac{j-1}{2k} \leq \fracpart{\xi a} < \frac{j}{2k} \right\} \text{ for $j = 1,\ldots,2k$,} \]
and similarly for $B_j$. We claim that if $\xi$ is chosen appropriately then any $A_r$ and $B_s$ will lead to a good $\psi$ above; the sets $A'$ and $B'$ can thus simply be taken to be the largest $A_r$ and $B_s$, which will clearly have sizes at least $\abs{A}/2k$ and $\abs{B}/2k$.

That $\psi$ is well-defined and a $k$-isomorphism on each $A_r+B_s$ will follow if we can show that there is a $\xi$ such that for arbitrary $a_1, \ldots, a_k, a'_1,\dots,a'_k \in A_r$ and $b_1, \ldots, b_k,b'_1,\dots,b'_k \in B_s$ the congruence
\begin{equation}
\begin{split}
&\floor{\xi a_1} + \cdots + \floor{\xi a_k} +\floor{\xi b_1} + \cdots + \floor{\xi b_k} \\
&\equiv 
\floor{\xi a'_1} + \cdots + \floor{\xi a'_k} +\floor{\xi b'_1} + \cdots + \floor{\xi b'_k} \pmod{N} \label{eqn:FreimanCong}
\end{split}
\end{equation}
holds if and only if the equality
\begin{align}
a_1 + \cdots + a_k + b_1 + \cdots + b_k
=a'_1 + \cdots + a'_k + b'_1 + \cdots + b'_k
\label{eqn:FreimanEquality}
\end{align}
holds. To see that such a $\xi$ exists, consider the quantity
\begin{equation}
\begin{split}
\sum_{i=1}^k \big( \floor{\xi a_i} - \floor{\xi a'_i} \big) 
&+\sum_{i=1}^k \big( \floor{\xi b_i} - \floor{\xi b'_i} \big) \\
= \xi\sum_{i=1}^k (a_i - a'_i) + \xi\sum_{i=1}^k (b_i - b'_i) 
- \sum_{i=1}^k \big(& \fracpart{\xi a_i} - \fracpart{\xi a'_i} \big)
- \sum_{i=1}^k \big( \fracpart{\xi b_i} - \fracpart{\xi b'_i} \big).
\end{split}
\label{eqn:modVsInt}
\end{equation}
Supposing first that \eqref{eqn:FreimanEquality} holds, we have that if all the $a_i,a'_i$ and $b_i,b'_i$ lie in $A_r$ and $B_s$ respectively then the last two sums in \eqref{eqn:modVsInt} have absolute value strictly less than $1/2$ each, and so the left-hand side is an integer lying strictly between $-1$ and $1$, and is hence $0$. Thus \eqref{eqn:FreimanCong} holds even as an equality of integers, not just a congruence, regardless of the choice of $\xi$.

Suppose now that \eqref{eqn:FreimanCong} holds. The left-hand side of \eqref{eqn:modVsInt} is then a multiple of $N$, and the right-hand side is of the form $\xi t + \delta$ where $t \in kA-kA+kB-kB$ and $\abs{\delta} < 1$. We want, in this notation, our choice of $\xi$ to force $t = 0$. It thus suffices to pick $\xi \in [0,N]$ such that
\[ \xi \notin \frac{1}{t}\cdot(dN + (-1,1)) \]
for any integer $d$ and non-zero integer $t \in kA-kA+kB-kB$. For a fixed $t$, this condition excludes at most $\abs{t}+1$ intervals from $[0,N]$ (corresponding to $d \in \{0,\ldots,t\}$), the lengths of which sum to $2$. The number of integers $t$ that need to be considered is $(\abs{kA-kA+kB-kB}-1)/2$, since we are omitting $0$ and since $t$ and $-t$ give rise to the same collection of excluded intervals. The sum total lengths of the excluded intervals is thus $\abs{kA-kA+kB-kB}-1$, and since $N \geq \abs{kA-kA+kB-kB}$ this means that there is a $\xi \in [0,N]$ outside these intervals, completing the proof.
\end{proof}

\begin{remark}
Clearly one can extend this statement to $m$ sets instead of two, as long as one is willing to reduce the size of the sets by a factor of $mk$ instead of $2k$.
\end{remark}

In order to apply this result we shall need an efficient bound on the size of the set $2A-2A+2B-2B$ when $A+B$ is small. The following lemma provides this.

\begin{lemma}\label{lemma:plunnecke}
Suppose $A$ and $B$ are sets of integers with $\abs{A+B} \leq K\min(\abs{A}, \abs{B})$. Then 
$\abs{2A-2A+2B-2B}\leq K^{C} \abs{B}$.
\end{lemma}
\begin{proof}
This is an application of standard tools from additive combinatorics. There are many possible ways to argue; here is one. By averaging, there is an element $x$ such that the set $D = (x-A) \cap B$ has size at least $\abs{A}\abs{B}/\abs{A+B} \geq \abs{A}/K$. By the Ruzsa triangle inequality \cite[Lemma 2.6]{tao-vu} we then have
\[ \abs{2A-2A+2B-2B} \leq \frac{\abs{2A-2A-D}\abs{D+2B-2B}}{\abs{D}} \leq \frac{\abs{3A-2A}\abs{3B-2B}}{\abs{D}} \leq K^{11} \abs{B}, \]
where at the last step we used the Pl\"unnecke-Ruzsa inequality \cite[Corollary 6.29]{tao-vu} (see also \cite{petridis} for Petridis's simple proof of this inequality). \end{proof}

\section{Improving the density dependence and using small doubling}\label{section:strong_lp}
As previously mentioned, in order to obtain the improved density dependence in Theorem \ref{thm:impAPs} we shall make use of the more detailed moment estimates given in \cite{croot-sisask}. These are encoded in the following specialization of \cite[Proposition 3.3]{croot-sisask}.

\begin{theorem}\label{thm:original_lp}
Let $p \geq 2$ and $k \in \N$ be parameters. Let $G$ be a finite abelian group and let $A$, $B$ and $S$ be subsets of $G$ with $\abs{A+S} \leq K_A\abs{A}$ and $\abs{A+B} \leq K_B\abs{B}$. Suppose $0 < \epsilon \leq k/\sqrt{K_B}$. Then there is a set $T \subset S$ of size
\[ \abs{T} \geq \exp(-C k^2 p \log(2K_A)/\epsilon^2) |S| \]
such that, for each $t \in kT-kT$,
\[ \norm{ \mu_A*1_B(x+t) - \mu_A*1_B(x) }_{L^p(x)} \leq \epsilon \norm{ \mu_A*1_B }_{p/2}^{1/2}. \]
\end{theorem}

(Recall that $\mu_A = 1_A/\mu_G(A)$.) The proof of this theorem in \cite{croot-sisask} was combinatorial, and various consequences of it were drawn by combinatorial means. It turns out that if we couple its conclusion with some relatively simple Fourier analysis then we can magnify its quantitative effectiveness dramatically: more specifically, following one of the key ideas of using Chang's theorem to handle a highly iterated convolution in the paper \cite{sanders:bogolyubov} of Sanders, we shall `bootstrap' the iterated sumset of almost-periods given by the previous theorem to a Bohr set of almost-periods.

\begin{theorem}\label{thm:strong_lp}
Let $G$ be a finite abelian group and let $A, B \subset G$ be non-empty subsets with $\abs{A+B} \leq K_A\abs{A}, K_B\abs{B}$. Let $p \geq 2$ and suppose $0 < \epsilon \leq 1/\sqrt{K_B}$. Then there is a Bohr set $T \subset G$ of rank $d \leq C p (\log 2K_A/\epsilon)^2 (\log 2K_A)/\epsilon^2 + C\log{\mu_G(B)^{-1}}$ and radius at least $c\epsilon/d\sqrt{K_A}$ such that
\[ \norm{ \mu_A*1_B(x+t) - \mu_A*1_B(x) }_{L^p(x)} \leq \epsilon \norm{ \mu_A*1_B }_{p/2}^{1/2} \text{ for each $t \in T$}. \]
\end{theorem}
\begin{proof}
Apply the previous theorem with $S = B$ and parameter $k$ to be determined to get a set $X$ with
\[ \tau := \abs{X}/\abs{G} \geq \exp\left(-C k^2 p (\log 2K_A)/\epsilon^2 - \log{ \mu_G(B)^{-1}} \right) \]
such that 
\[ \norm{ \mu_A*1_B(x+t) - \mu_A*1_B(x) }_{L^p(x)} \leq \tfrac{1}{3}\epsilon \norm{ \mu_A*1_B }_{p/2}^{1/2} \text{ for each $t \in k X$}. \]
By the triangle inequality we thus have
\[ \norm{ \mu_A*1_B*\mu_X^{(k)} - \mu_A*1_B }_p \leq \tfrac{1}{3}\epsilon \norm{ \mu_A*1_B }_{p/2}^{1/2}, \]
where $\mu_X^{(k)} = \mu_X*\mu_X*\cdots*\mu_X$ denotes the $k$-fold convolution of $\mu_X$ with itself.

Now let $t \in G$ be arbitrary. Another application of the triangle inequality yields
\begin{equation}
\begin{split}
\norm{ \mu_A*1_B(x+t) &- \mu_A*1_B(x) }_{L^p(x)} \leq \tfrac{2}{3}\epsilon \norm{ \mu_A*1_B }_{p/2}^{1/2} \\
&+ \norm{ \mu_A*1_B*\mu_X^{(k)}(x+t)1_S(x) - \mu_A*1_B*\mu_X^{(k)}(x)1_S(x)}_{L^p(x)},
\end{split}\label{eqn:lp_bound}
\end{equation}
where $S = (A+B) \cup (A+B-t)$ contains the support of $\mu_A*1_B(x+t) - \mu_A*1_B(x)$. We shall bound the last term here by restricting $t$ to an appropriate Bohr set. Indeed, bounding the term by its maximum on $S$ and then using Fourier inversion we have that it is at most
\begin{align*}
2^{1/p} &\mu_G(A+B)^{1/p} \norm{ \mu_A*1_B*\mu_X^{(k)}(x+t) - \mu_A*1_B*\mu_X^{(k)}(x)}_{L^\infty(x)} \\
&\leq 2^{1/p} \mu_G(A+B)^{1/p} \sum_{\gamma \in \Ghat} \abs{\widehat{\mu_A}(\gamma)} \abs{\widehat{1_B}(\gamma)} \abs{\widehat{\mu_X}(\gamma)}^k \abs{\gamma(t) - 1}.  
\end{align*}
Now let $\Gamma = \{ \gamma \in \Ghat : \abs{\widehat{\mu_X}(\gamma)} \geq 1/e \}$ and set $\delta = \epsilon/5\sqrt{K_A}$ and $k = \ceiling{ \log 2/\delta }$. Then, for any $t \in \Bohr_G(\Gamma, \delta)$ we have
\begin{align*}
\sum_{\gamma \in \Ghat} \abs{\widehat{\mu_A}(\gamma)} \abs{\widehat{1_B}(\gamma)} \abs{\widehat{\mu_X}(\gamma)}^k \abs{\gamma(t) - 1} 
&\leq \delta \sum_{\gamma \in \Gamma} \abs{\widehat{\mu_A}(\gamma)} \abs{\widehat{1_B}(\gamma)} + \frac{2}{e^k} \sum_{\gamma \notin \Gamma} \abs{\widehat{\mu_A}(\gamma)} \abs{\widehat{1_B}(\gamma)} \\
&\leq \delta \mu_G(B)^{1/2}\mu_G(A)^{-1/2}.
\end{align*}
To be of use in \eqref{eqn:lp_bound} we thus want $2^{1/p} \mu_G(A+B)^{1/p} \delta \mu_G(B)^{1/2}\mu_G(A)^{-1/2}$ to be at most $\tfrac{1}{3}\epsilon \norm{ \mu_A*1_B }_{p/2}^{1/2}$. A quick calculation using H\"older's inequality and the relationship $\norm{\mu_A*1_B}_1 = \mu_G(B)$ reveals that our choice of $\delta$ ensures this.
Hence we have the conclusion we want for any $t \in \Bohr_G(\Gamma, \delta)$. To obtain the low rank conclusion, we apply Chang's theorem \cite{chang} (see also \cite[Lemma 4.36]{tao-vu}): this says that $\Bohr_G(\Gamma, \delta)$, being a Bohr set associated to a set of large Fourier coefficients, contains a low-rank Bohr set $\Bohr_G(\Lambda, \delta/d)$ where
\[ d := \abs{\Lambda} \leq C \log{1/\tau} \leq C p (\log 2K_A/\epsilon)^2 (\log 2K_A)/\epsilon^2 + C\mu_G(B)^{-1}, \]
which completes the proof.
\end{proof}

\begin{proof}[Proof of Theorem \ref{thm:impAPsDoubling}]
We first need to embed the sets $A$ and $B$ into a finite abelian group so that we can effectively use the properties of Bohr sets. As mentioned in Section \ref{section:model}, this is simple in the setup of Theorem \ref{thm:impAPs}, where $A$ and $B$ are dense subsets of $\{1,\ldots,N\}$: we may simply embed the sets in $\Zmod{N'}$ where $N'$ is a prime between $4N$ and $8N$. For the small doubling case we use Proposition \ref{prop-ruzsa}. Specifically, let us assume without loss of generality that $K_B \leq K_A$. Applying Proposition \ref{prop-ruzsa} to $A$ and $B$ with $k=2$ and $N$ a prime less than $CK_A^C \abs{B}$, as allowed by Lemma \ref{lemma:plunnecke}, we obtain sets $A' \subset A$, $B' \subset B$ with $\abs{A'} \geq \abs{A}/4,$ $\abs{B'} \geq \abs{B}/4$ and a Freiman $2$-isomorphism $\phi$ from $A'+B'$ to a subset of $\Zmod{N}$. We shall find our desired arithmetic progression in $A'+B'$. We may assume, by translating if necessary, that $A'$ and $B'$ both contain $0$ and that $\phi$ takes $0$ to $0$. Hence we have $\phi(A'+B') = \phi(A') + \phi(B')$ and so $\abs{\phi(A') + \phi(B')} \leq 4K_A\abs{\phi(A')}, 4K_B\abs{\phi(B')}$. Since the property of containing an arithmetic progression of a given size is preserved under $2$-isomorphism, it suffices (after adjusting the constants) to establish the theorem for $\phi(A')$ and $\phi(B')$, which we now relabel as $A$ and $B$.

The argument is now similar to that in the proof of Theorem \ref{thm:green} given in Section \ref{section:prob_arg}, the main difference being that we take into account the higher-order energy $\norm{\mu_A*1_B}_{p/2}$ present in the conclusion of Theorem \ref{thm:strong_lp}. Let us apply Theorem \ref{thm:strong_lp} with $\epsilon = 1/e\sqrt{K_B}$ and a parameter $p$ to be specified later; since $B$ has density at least $cK_A^{-C}$ in $\Zmod{N}$ this gives us a Bohr set $T$ of rank $d \leq Cp K_B(\log 2K_A)^3$ and radius at least $c/K_A d$ such that 
\[ \norm{ \mu_A*1_B(x+t) - \mu_A*1_B(x) }_{L^p(x)} < \epsilon \norm{ \mu_A*1_B }_{p/2}^{1/2} \text{ for each $t \in T$}. \]
Let $P$ be an arithmetic progression in $T$. Assume for a contradiction that $A+B$ does not contain a translate of $P$, so that for each $x \in \Zmod{N}$ there is some $t \in P$ for which $\mu_A*1_B(x+t) = 0$. Then
\begin{align*}
\abs{P} \epsilon^p \norm{ \mu_A*1_B }_{p/2}^{p/2} &> \sum_{t \in P} \norm{\mu_A*1_B(x+t) - \mu_A*1_B(x)}_{L^p(x)}^p \\
&\geq \norm{\mu_A*1_B}_p^p \geq \norm{\mu_A*1_B}_{p/2}^p/\mu_G(A+B),
\end{align*}
the last inequality being an instance of the Cauchy-Schwarz inequality.
We shall thus have a contradiction if $\abs{P}$ and $p$ are picked so that
\[ \abs{P} \leq e^p K_B^{p/2} \norm{\mu_A*1_B}_{p/2}^{p/2}/\mu_G(A+B). \] 
Since $\norm{\mu_A*1_B}_{p/2} \geq \norm{\mu_A*1_B}_1/\mu_G(A+B)^{1-2/p}$ by H\"older's inequality, it suffices to pick $\abs{P} \leq e^p$ to obtain the contradiction. Now, by Lemma \ref{lemma:bohr_sets} we can find an arithmetic progression $P$ in $T$ of size the integer part of
\[ \exp\left( c \frac{\log{N}}{p K_B (\log{2K_A})^3} - C\log(p K_A) \right); \]
picking $p = C\sqrt{\log{N} / K_B (\log{2K_A})^3}$ then yields the conclusion, since $\abs{A+B} \leq N \leq CK_A^C \abs{A}$.
\end{proof}

\section{Further remarks}\label{section:remarks}

\subsection*{Vector spaces over finite fields}
A very useful philosophy in additive combinatorics, particularly expounded and well-exposited by Green \cite{green:finite_fields}, is that it can often be simpler to study a given additive problem perhaps not in its original setting but rather in a particular family of groups, namely the vector spaces over finite fields $\F_q$. One of the simplifications that this setting allows over that of general abelian groups is the replacement of Bohr sets with subspaces, which are exact rather than approximate annihilators of sets of characters, and this makes many arguments simpler; this is the case in \cite{green:longAPs} and \cite{sanders:longAPs}, for example. Breaking the trend, the methods of this paper hardly change when looked at in vector spaces. The only significant change occurs in the proofs of Theorems \ref{thm:green} and \ref{thm:impAPs}: instead of taking $P$ to be an arithmetic progression in the Bohr set $T$, one can let $P$ be a \emph{subspace} of the \emph{subspace} $T$. Making this change gives the following analogue of Theorem \ref{thm:green}, a result originally established (in the $q=2$ case) by Green \cite{green:restrictionKakeya} and then by Sanders \cite{sanders:longAPs}.

\begin{theorem}
Suppose $A$ and $B$ are subsets of $\F_q^n$ of densities $\alpha$ and $\beta$. Then $A+B$ contains a translate of a subspace of dimension at least $c \alpha \beta n / \log q$.
\end{theorem}

The argument used in the proof of Theorem \ref{thm:impAPs} similarly yields the following version with improved density dependence.

\begin{theorem}\label{thm:improved_density_Fpn}
Suppose $A$ and $B$ are subsets of $\F_q^n$ of densities $\alpha$ and $\beta$. Then $A+B$ contains a translate of a subspace of dimension at least $c \alpha (\log 2\beta^{-1})^{-3} n / \log q$.
\end{theorem}

In the case when $B=-A$ it is known from further work of Sanders \cite{sanders:GreenHalf} that one can remove the logarithmic factor. It has furthermore been suggested by Green \cite{green:finite_fields} that the true nature of the bound should be one of an upper bound on codimension rather than a lower bound on dimension, so Theorem \ref{thm:improved_density_Fpn} might be quite far from the truth.

One need not specialize $P$ to be a subspace in the proofs, of course; the more general result one obtains is then the following.

\begin{theorem}
Suppose $A$ and $B$ are subsets of $\F_q^n$ of densities $\alpha$ and $\beta$. Then, for any $d \in \N$, there is a subspace $V \leq \F_q^n$ of codimension $d$ such that $A+B$ contains a translate of any subset of $V$ of size at most $\exp\left( c \alpha (\log 2\beta^{-1})^{-3} d \right)$.
\end{theorem}

\subsection*{Remarks on the bootstrapping procedure}

Exactly the same smoothing-by-convolution argument that deduced Theorem \ref{thm:strong_lp} from Theorem \ref{thm:original_lp} can be used to deduce the following result from Theorem \ref{thm:plain_lp}.

\begin{theorem}\label{thm:plain_lp_bohr}
Let $p \geq 2$ and $\epsilon \in (0,1)$. Let $G$ be a finite abelian group and let $A$, $B$ and $S$ be finite subsets of $G$ with $\abs{A+S} \leq K_1\abs{A}$ and $\abs{A+B} \leq K_2 \abs{B}$. Then there is a Bohr set $T$ of rank at most 
$$d=Cp(\log(1/\delta))^2\epsilon^{-2}\log (2K_1) + C\log(\mu_G(S)^{-1})$$
and radius at least $\delta/d$, where
$$\delta=c\epsilon\sqrt{\mu_G(A)/\mu_G(B)}K_2^{-1/p},$$
such that, for each $t \in T$,
\[ \norm{ \mu_A*1_B(x+t) - \mu_A*1_B(x) }_{L^p(x)} \leq \epsilon \mu_G(B)^{1/p}. \]
\end{theorem}

Although considerably more technical looking, this theorem has the advantage over Theorem \ref{thm:plain_lp} that the set of almost-periods it gives is more easily seen to be structured than the iterated sumset given by Theorem \ref{thm:plain_lp}. 
Applying this with $S=A$, $B=A-A$, $\epsilon = 1/2$ and $p=C\log(2K)$, where $K$ is the doubling constant for $A$, it is a quick deduction to recover Theorems 11.1 and 11.3 of Sanders \cite{sanders:bogolyubov}, namely that $2A-2A$ contains a low-rank coset progression in groups with `good modelling', like $\Z$ and $\F_p^n$. Indeed, for any almost-period $t \in T$ one has
\begin{align*}
\abs{\mu_{-A}*\mu_A*1_{A-A}(t) - 1} &= \abs{\mu_{-A}*\mu_A*1_{A-A}(t) - \mu_{-A}*\mu_A*1_{A-A}(0)} \\
&\leq \norm{\mu_A}_{L^q} \norm{\mu_A*1_{A-A}(x+t) - \mu_A*1_{A-A}(x)}_{L^p(x)} \\
&\leq c,
\end{align*}
whence $t \in 2A-2A$. Since $T$ contains a large coset progression, one is done.  We stress that this does not offer a quantitative improvement over Sanders's result, but we include the remark for clarification and since it follows a slightly different and more direct route to that in \cite{sanders:bogolyubov}, avoiding the need to deal with correlations and deducing containment from this.
We finish by recording the rather clean `density version' of the result one obtains in this fashion:

\begin{theorem}\label{thm:tom_bogolyubov}
Let $G$ be a finite abelian group and suppose $A \subset G$ has size $\alpha \abs{G}$. Then $2A-2A$ contains a Bohr set of radius at least $c \alpha^{1/2}$ and rank at most $C(\log 1/\alpha)^4$.
\end{theorem}

\end{document}